\documentclass{aptpub}
\usepackage{amsfonts}
\usepackage{dsfont}

\authornames{Ge Chen, Changlong Yao and Tiande Guo} % insert the authors here for use in running head
\shorttitle{The Asymptotic Size of The Largest Component in RGGs} % insert short title here for use in running head
\usepackage{graphicx}
\begin{document}

\title{The Asymptotic Size of The Largest Component in Random Geometric Graphs with some applications} % insert title - use \\ if it requires more than one line.

%\author{Ge Chen\footnote{First author. Email: chenge@amss.ac.cn}, Changlong Yao\footnote{Second author. Email: deduce04@mails.gucas.ac.cn} and Tiande Guo\footnote{Corresponding author. Email:tdguo@gucas.ac.cn}
%\small{School of Mathematical Science, Graduate University of
%Chinese Academy of Sciences}

\authorone{Ge Chen}
\addressone{National Center for Mathematics and Interdisciplinary Sciences \& Key Laboratory of Systems and Control, Academy of Mathematics and Systems Science,
Chinese Academy of Sciences, Beijing, 100190, P.R.China. Email: chenge@amss.ac.cn.} % Your postal address goes here.
\authortwo[Academy of Mathematics and Systems Science,
CAS]{Changlong Yao} % Affiliation is just the name of your university or institution
\addresstwo{Institute of Applied Mathematics, Academy of Mathematics and Systems Science,
Chinese Academy of Sciences, Beijing, 100190, P.R.China. Email: deducemath@126.com.}
\authorthree[Graduate University of
Chinese Academy of Sciences]{Tiande Guo} % Affiliation is just the name of your university or institution
\addressthree{School of Mathematical Science, Graduate University of
Chinese Academy of Sciences, Beijing, 100049, P.R.China. Email: tdguo@gucas.ac.cn.}

\begin{abstract}
For the size of the largest component in a supercritical random geometric graph,
this paper estimates its expectation which tends to a polynomial on a rate of exponential decay, and  sharpens its
 asymptotic  result with a central limit theory. Similar results can be obtained for the size
of biggest open cluster, and for the number of open clusters of
percolation on a box, and so on.
\end{abstract}

 % insert keywords separated by a semicolon

 % insert the primary Maths Subject Classification number in the first bracket
         % and the secondary ams number(s) in the second bracket
         % e.g. \ams{60E20}{49G03;49F10}

\section{Introduction} % Initial capital letter, then lower case. No full stop.
The size of the largest component is a basic property for random
geometric graphs (RGGs) and has attracted much interest during the
past years, including both theoretical
studies~\cite{penrose1995}\cite{penrose1996}\cite{penrose2001}\cite{p10}
and various
applications~\cite{Glauche2003}\cite{p17}\cite{p16}\cite{Pishro2009}.
This paper firstly investigates the
asymptotic size of the largest component of RGG in the supercritical
case.

Given a set $\mathcal {X}\subset\mathbb{R}^d$, let $G(\mathcal
{X};r)$ denote the undirected graph with vertex set $\mathcal {X}$
and with undirected edges which connect all those pairs $\{X,Y\}$ with
$\parallel Y-X\parallel\leq r$, where $\|\cdot\|$ denotes the
Euclidean norm ($l_2-norm$). The basic model of RGGs can be
formulated as $G(\mathcal {X}_n;r_n)$, where $\mathcal{X}_n$ denotes
 $n$ points which are independently and uniformly distributed in a
$d$-dimensional unit cube. To overcome the lack of spatial
independence for the binomial point process $\mathcal{X}_n$, the
model of continuum percolation must be introduced. Following Section
1.7 in  \cite{p10}, let $\mathcal{H}_{\lambda}$ be a homogeneous
Poisson process of intensity $\lambda$ on $\mathbb{R}^{d}$. For
$s>0$,  define $B(s):=[0,s]^{d}$ and
$\mathcal{H}_{\lambda,s}:=\mathcal{H}_{\lambda}\cap B(s).$ Following
\cite{p10}, we write the Poisson Boolean model as $G(\mathcal
{H}_{\lambda,s};1)$.

There exist some notations related to percolation must be introduced.
Following Section 9.6 in \cite{p10}, let $\mathcal{H}_{\lambda,0}$
denote the point process $\mathcal{H}_{\lambda}\cup\{\textbf{0}\}$,
where $\textbf{0}$ is the origin in $\mathbb{R}^d$, and for $k\in\mathds{N}$, let $p_{k}(\lambda)$ denote the probability that
the order of the component in $G(\mathcal{H}_{\lambda,0};1)$ containing the
origin is equal to $k$. The $percolation$ $probability$
$p_{\infty}(\lambda)$ is defined to be the probability that $\textbf{0}$ lies in
an infinite component of the graph $G(\mathcal{H}_{\lambda,0};1)$.
Therefore, we have
$p_{\infty}(\lambda)=1-\sum\limits_{k=1}^{\infty}p_{k}(\lambda)$.
Let
\begin{eqnarray}\label{lc}
\lambda_c=inf\{\lambda>0:p_{\infty}(\lambda)>0\}
\end{eqnarray}
denote the  critical intensity  of continuum percolation. It is well known that
$0<\lambda_c<\infty$ for $d\geq 2$~\cite{p12}\cite{p13}\cite{p14}.

Following Section 9.6 in \cite{p10}, let $L_j(G)$ denote the order
of its $j$th-largest component for any graph $G$. Then
$L_1(G(\mathcal {H}_{\lambda,s};1))$ denotes the order of the
largest component of $G(\mathcal {H}_{\lambda,s};1)$. The asymptotic
properties of $L_1(G(\mathcal{H}_{\lambda,s};1))$ have been well
studied by Penrose. The basic asymptotic result about $L_1
(G(\mathcal{H}_{\lambda,s};1))$ is provided by Penrose (Theorem 10.9
in \cite{p10}), that if $\lambda\neq\lambda_c$ then
\begin{eqnarray}\label{o1}
s^{-d}L_1(G(\mathcal{H}_{\lambda,s};1))\xrightarrow{P}\lambda
p_{\infty}(\lambda)\quad as \quad s\rightarrow \infty.
\end{eqnarray}
Also, Penrose has given a central limit theorem for
$L_1(G(\mathcal{H}_{\lambda,s};1))$ in the supercritical case
$\lambda>\lambda_c$ (Theorem 10.22 in \cite{p10}), that
\begin{eqnarray}\label{clt_1}
s^{-d/2}(L_1(G(\mathcal{H}_{\lambda,s};1))-E[L_1(G(\mathcal{H}_{\lambda,s};1))])
\xrightarrow{D} \mathcal {N}(0,\sigma^2).
\end{eqnarray}
 However, the question as how large
$E[L_1(G(\mathcal{H}_{\lambda,s};1))]$ should be still remains
unsolved. By (\ref{o1}) it can be deduced that
$E[L_1(G(\mathcal{H}_{\lambda,s};1))]=\lambda p_{\infty}(\lambda)
s^d +o(s^d)$, where $f(s)=o(g(s))$ indicates that $\lim_{s\rightarrow\infty}\frac{f(s)}{g(s)}=0$. This result is not precise enough for some theoretic analysis
and practical applications.

The corresponding asymptotic results and central limit theorem for
$G(\mathcal {X}_n;r_n)$ have also been established by Peorose
(Theorems 11.9 and 11.16 in \cite{p10}), but we may ask
similar questions. This paper will study the problem and
 give a more precise description for
the asymptotic sizes of $L_1(G(\mathcal{H}_{\lambda,s};1))$ and
$L_1(G(\mathcal {X}_n;r_n))$. Our method can be adapted to study
some other models and problems.

\section{Main Results}

Our main  results can be formulated as the following two theorems.
\begin{thm}\label{t1}
Suppose $d\geq 2$ and $\lambda>\lambda_c$. Then there exist
constants $c=c(d,\lambda)>0$ and
$\tau_i=\tau_i(d,\lambda)$, $1\leq i \leq d$, with $\tau_1>0$, such that for all $s$
large enough,
\begin{eqnarray}\label{order_t1_00}
E[L_1(G(\mathcal{H}_{\lambda,s};1))]=\lambda
p_{\infty}(\lambda) s^d-\sum_{i=1}^d \tau_i s^{d-i}+o\left(e^{-c s}\right).
\end{eqnarray}
Also,  there exists a constant $\sigma=\sigma(d,\lambda) > 0$, such
that
\begin{eqnarray}\label{order_t1_01}
L_1(G(\mathcal{H}_{\lambda,s};1))s^{-d/2}- \lambda
p_{\infty}(\lambda) s^{d/2}+ \sum_{i=1}^{\lfloor\frac{d}{2} \rfloor}\tau_i s^{d/2-i}
 \xrightarrow{D} \mathcal {N}(0,\sigma^2)
\end{eqnarray}
as $s\rightarrow\infty$.
\end{thm}

\begin{thm}\label{t2}
Suppose $d\geq 2$ and $\lambda>\lambda_c$. Let $\sigma$ and $\tau_i$
be the same constants appearing in Theorem \ref{t1}. There exists a
constant $\delta=\delta (d,\lambda)$, with $0<\delta\leq\sigma$,
such that
\begin{eqnarray*}
L_1\left(G\left(\mathcal{X}_n;(n/\lambda)^{-1/d}\right)\right)\left(n/\lambda\right)^{-1/2}-
p_{\infty}(\lambda) \left(\lambda n\right)^{1/2}+\sum_{i=1}^{\lfloor\frac{d}{2} \rfloor} \tau_i
\left(n/\lambda\right)^{\frac{1}{2}-\frac{i}{d}}
 \xrightarrow{D} \mathcal {N}(0,\delta^2)
\end{eqnarray*}
as $n\rightarrow\infty$.
\end{thm}

To prove the two theorems, we estimate the value of
$E[L_1(G(\mathcal{H}_{\lambda,s};1))]$ firstly, and then using the
central limit theorems for $L_1(G(\mathcal{H}_{\lambda,s};1))$ and
$L_1(G(\mathcal{X}_n;(n/\lambda)^{-1/d}))$, we can prove
(\ref{order_t1_01}) and Theorem \ref{t2}.

Some notations must be stated before the proof of
our results. For any $x\in \mathbb{R}^d$, we write its $l_{\infty}$
norm with $\|x\|_{\infty}$ given by the maximum absolute value of
its coordinates. For any finite set $A\subset \mathbb{R} ^d$, we set
the diameter of $A$ by diam$(A)=\sup_{x,y\in A}\|x-y\|_{\infty}.$
Also, let $|A|$ denote the cardinality of $A$.

Let $\oplus$ denote the Minkowski addition of sets. Let $Leb(\cdot)$
denote the Lebesgue measure. For $s \geq 0$,   let $\lfloor s
\rfloor$ denote the smallest integer not smaller than $s$.

To simplify the expression, we will omit the dependence of all
constants on $d$ and $\lambda$, for example, the constant $c$ stands
for $c(d,\lambda)$.

Given $\lambda>\lambda_c$, by the uniqueness of the infinite
component in continuum percolation (Theorem 9.19 in \cite{p10}), the
infinite graph $G(\mathcal{H}_{\lambda};1)$ has precisely one
infinite component $\mathcal {C}_{\infty}$ with probability $1$.
%Let $\mathcal {C}_{\infty}$ denote the point set of this infinite component.
Let $C_1,C_2,...,C_M$ denote the components of $G(\mathcal
{C}_{\infty}\cap B(s);1)$, taken in a decreasing order. We give a
result on the rate of sub-exponential decay of the difference
between $E[L_1(G(\mathcal{H}_{\lambda,s};1))]$ and $E[|C_1|]$.

\begin{lem}\label{temp1}
Suppose $d\geq 2$ and $\lambda>\lambda_c$. The exists a constant
$c>0$, such that for large enough $s$,
\begin{eqnarray}\label{temp1_0}
0\leq E[L_{1}(G(\mathcal{H}_{\lambda,s};1))]-E[|C_1|] \leq e^{-cs}.
\end{eqnarray}
\end{lem}

\begin{proof}
By the definition of $L_{1}(G(\mathcal{H}_{\lambda,s};1))$ and
$C_1$, obviously $E[L_{1}(G(\mathcal{H}_{\lambda,s};1))]\geq
E[|C_1|]$. Thus it just remains to prove the second inequality of
(\ref{temp1_0}).

 Given any $x\in\mathbb{R}^d$, let
$C_{\infty}(x)$ denote the infinite connected component of
$G(\mathcal{H}_{\lambda}\cup\{x\};1)$. By Palm theorem for Poisson
processes (Theorem 1.6 in \cite{p10}), we have
\begin{eqnarray*}\label{order_t1_6}
E[L_{1}(G(\mathcal{H}_{\lambda,s};1))]=\lambda\int_{B(s)}P[x\in
V_{1}(x)]dx,
\end{eqnarray*}
where $V_{1}(x)$ denotes the largest component of
$G(\mathcal{H}_{\lambda,s}\cup\{x\};1)$, and
\begin{eqnarray*}\label{order_t1_7}
E[|C_1|]=\lambda\int_{B(s)}P[x\in C_{1}(x)]dx,
\end{eqnarray*}
where $C_{1}(x)$ denotes the largest component of $C_{\infty}(x)
\cap B(s)$. Therefore,
\begin{eqnarray}\label{order_t1_8}
\begin{aligned}
E[L_{1}(G(\mathcal{H}_{\lambda,s};1))]-E[|C_1|]&=\lambda\int_{B(s)}(P[x\in
V_{1}(x)]-P[x\in C_{1}(x)])dx\\
&\leq \lambda\int_{B(s)} P[\{x\in V_{1}(x)\} \cap \{ x\notin
C_1(x)\}]dx\\
&= \lambda\int_{B(s)} P[\{x\in V_{1}(x)\} \cap \{ x\notin
C_{\infty}(x)\}]dx.
\end{aligned}
\end{eqnarray}
Suppose $0<\varepsilon<\frac{1}{2}$. By Theorem 10.19 in \cite{p10},
there exist constants $c_{1}>0$ and $s_{1}>0$, such that if $s>s_1$
then
\begin{eqnarray}\label{order_t1_9}
\begin{aligned}
P\left [|V_{1}(x)|<(1-\varepsilon)\lambda s^{d}p_{\infty}(\lambda)
\right ] &\leq P\left
[L_{1}(G(\mathcal{H}_{\lambda,s};1))<(1-\varepsilon)\lambda
s^{d}p_{\infty}(\lambda)\right ]\\
&\leq \exp \left (-c_{1}s^{d-1} \right ).
\end{aligned}
\end{eqnarray}
Also, by Theorem 10.15 in \cite{p10}, there exists a constant
$c_2>0$ such that for $s$ large enough,
\begin{eqnarray}\label{order_t1_10}
\sum\limits_{k\geq \lceil(1-\varepsilon)\lambda
s^{d}p_{\infty}(\lambda)\rceil}p_{k}(\lambda)<\exp \left
(-c_2[(1-\varepsilon)\lambda s^{d}p_{\infty}(\lambda)]^{(d-1)/d}
\right ).
\end{eqnarray}
Therefore, from (\ref{order_t1_9}) and (\ref{order_t1_10}) we can
obtain
\begin{eqnarray*}\label{order_t1_11}
&&P[\{x\in V_{1}(x)\}\cap\{x\not\in
C_{\infty}(x)\}]\nonumber\\
&&~~~~\leq P[|V_{1}(x)|<(1-\varepsilon)\lambda
s^{d}p_{\infty}(\lambda)]\nonumber\\
&&~~~~~~~~~+P[\{x\in V_{1}(x)\}\cap\{x\not\in
C_{\infty}(x)\}\cap\{|V_{1}(x)|\geq(1-\varepsilon)\lambda
s^{d}p_{\infty}(\lambda)\}]\nonumber\\
&&~~~~\leq \exp \left (-c_{1}s^{d-1} \right )+\sum\limits_{k\geq
\lceil(1-\varepsilon)\lambda
s^{d}p_{\infty}(\lambda)\rceil}p_{k}(\lambda)\nonumber\\
&&~~~~< \exp \left (-c_{1}s^{d-1} \right )+ \exp \left
(-c_2[(1-\varepsilon)\lambda p_{\infty}(\lambda)]^{(d-1)/d} s^{d-1}
\right )~~as~s\rightarrow\infty.
\end{eqnarray*}
Combined with (\ref{order_t1_8}) this yields our result.
\end{proof}

To estimate the value of $E[L_1(G(\mathcal{H}_{\lambda,s};1))]$, by
Lemma \ref{temp1} we just need to get the value of $E[|C_1|]$
instead. Actually, by Palm theory for infinite Poisson process
(Theorem 9.22 in \cite{p10}),
\begin{eqnarray}\label{comm_6}
E\left [\sum_{i=1}^M |C_i| \right ]=E[\left|\mathcal
{C}_{\infty}\cap B(s)\right|]=\lambda p_{\infty}(\lambda)s^d,
\end{eqnarray}
so we just need to estimate the value of $E[\sum_{i=2}^M |C_i|]$.
Let $L(s):=B(s)\backslash[1,s-1]^d$. For any $2\leq i \leq M$, since
$C_i\subset \mathcal {C}_{\infty}$, therefore there exists at least
one point in $L(s)\cap C_i$ which connects to $\mathcal
{C}_{\infty} \setminus B(s)$ directly; we choose the nearest one to the
boundary of $B(s)$ as the $out-connect$ $point$. We can see that
each component of $C_2,...,C_M$ contains exactly one out-connect
point.

For any region $R\subseteq B(s)$ and $2\leq i \leq M$, define
\begin{equation}\label{chidef}
\chi_{i}(R):=\left\{
\begin{array}{ll}
1, & \mbox{if the out-connect point of $C_i$ is contained by } R, \\
0, & \mbox{otherwise},
\end{array}
\right.
\end{equation}
and define
\begin{equation}\label{xidef}
\xi(R)=\xi(R,s):= \sum\limits_{i=2}^{M}\chi_{i}(R)|C_i|.
\end{equation}
By the definition of $\xi(\cdot)$, it is easy to see that for any
$R,\widetilde{R}\subset B(s)$, if $Leb(R\cap \widetilde{R})=0$, then
$E[\xi(R\cap \widetilde{R})]=0$ and $E[\xi(R\cup
\widetilde{R})]=E[\xi(R)]+E[\xi(\widetilde{R})].$

For $0 \leq i \leq d-1$, define
$$R_i=R_i(s):=[0,1]\times\underbrace{[0,s/2]\times\cdots\times[0,s/2]}\limits_{d-1-i} \times \underbrace{[1,s/2]\times\cdots\times[1,s/2]}\limits_{i}.$$
Noted that $[1,s/2]^d \cap L(s)=\emptyset$, then by symmetry,
\begin{eqnarray}\label{total_1}
\begin{aligned}
&E\left [\sum\limits_{i=2}^M |C_i| \right
]=E\left[\xi(B(s))\right]=2^dE\left[\xi\left(\left[0,\frac{s}{2}\right]^d\right)\right]\\
&=2^d\left\{E\left[\xi(R_0)\right]+E\left[\xi\left(\left[1,\frac{s}{2}\right]\times\left[0,\frac{s}{2}\right]^{d-1}\right)\right]\right\} = 2^d\sum\limits_{i=0}^{d-1} E\left[\xi(R_i)\right].
\end{aligned}
\end{eqnarray}
Thus, we just need to estimate the value of $E\left[
\xi\left(R_i \right) \right]$. The following Lemmas
\ref{exponent2}-\ref{limit} are introduced to get the desired
estimation.

\begin{lem}\label{exponent2}
Suppose $d\geq 2$ and $\lambda>\lambda_c$. Let $V_x=V_x(s)$ denote
the connected component containing $x$ of
$G(\mathcal{H}_{\lambda,s}\cup \{x\};1)$. There exist constants
$c>0$ and $n_0>0$, such that if $n>n_0$ and $s>2n$ then for any point $x \in B(s)$,
\begin{eqnarray}\label{exp00}
P \left [ n \leq
\mbox{diam}(V_x) \leq s/2 \right ] < e^{ -c n },
\end{eqnarray}
and
\begin{eqnarray}\label{exp01}
P \left [ \left \{ |V_x| \geq n \right \} \cap \left \{
\mbox{diam}(V_x) \leq s/2 \right \} \right ] < \exp \left ( -c
n^{(d-1)/d} \right ).
\end{eqnarray}
\end{lem}
\begin{proof}
The proof uses ideas from the latter part of the proof of Theorem
10.18 in \cite{p10}. Given $x\in \mathbb{R}^d$, let $\widetilde{z}$
denote the point in $B_{\mathbb{Z} }'(n(s))$ satisfying $x\in
B_{\widetilde{z}}$, where the definition of $B_{\mathbb{Z} }'(n(s))$ and $
B_{\widetilde{z}}$ is given in pp.216 and pp.217 of \cite{p10} respectively. Also, $C_x$, $D_{ext}C_x$, $M_0$, $n(s)$ and $M(s)$ are defined as same as those appearing in pp.218-219 of
\cite{p10}. Penrose  has proved that $D_{ext}C_x$ is $*-$connected and if $|C_x|<n(s)^d/2$ then
\begin{eqnarray}\label{exp22_1}
\begin{aligned}
|D_{ext}C_x| \geq (2d)^{-1} (1- ({\textstyle  \frac{2}{3}})^{1/d})|C_x|^{(d-1)/d},
\end{aligned}
\end{eqnarray}
see pp.219 of \cite{p10}.

Let $\mathcal{A}_{m,s}$ denote the collection of $*-$connected
subsets of cardinality $m$ which disconnects the point
$\widetilde{z}$ from the giant component of $B_{\mathbb{Z}
}'(n(s))$. Then $\mathcal{A}_{m,s}$ is restricted by the box of
$B_{\mathbb{Z} }'(n(s))\cap ([-m,m]^d \oplus \widetilde{z})$ and
$D_{ext}C_x \in \mathcal{A}_{|D_{ext}C_x|,s}$. By a Peierls argument
(Corollary 9.4 in \cite{p10}), the cardinality $|\mathcal{A}_{m,s}|$
is bounded by $(2m+1)^d \gamma^m$, with $\gamma:=2^{3^d}$.
Therefore, there exists a constant $k_0$ such that for any integer
$k>k_0$,
\begin{eqnarray}\label{exp23}
\begin{aligned}
P \left [|D_{ext}C_x|\geq k \right ] &\leq P \left [
\bigcup\limits_{m \geq k} \bigcup\limits_{\sigma \in
\mathcal{A}_{m,s} } \{X_z=0, \forall z\in\sigma\} \right
]\\
&\leq \sum\limits_{m \geq k } (2m+1)^d \gamma^m (1-p_1)^m <
(\frac{2}{3})^{k}.
\end{aligned}
\end{eqnarray}

By the definition of $C_x$ and $D_{ext}C_x$,  if $n\leq$diam$(V_x)\leq
s/2$ then $$\frac{n}{M(s)}-1\leq\mbox{diam}(C_x)\leq \frac{n(s)}{2}+2,$$ and therefore we can get $|C_x| < n(s)^d/2$ and
$|D_{ext}C_x|\geq \frac{n}{M(s)}-1$ for large $s$.  Therefore, by (\ref{exp23}),
there exists a constant $n_0>0$, such that if $n>n_0$ then,
\begin{eqnarray*}
P \left [ n\leq
\mbox{diam}(V_x) \leq \frac{s}{2}\right ] \leq P \left
[|D_{ext}C_x|\geq \frac{n}{M(s)}-1 \right ]
 <\left (\frac{2}{3} \right )^{\frac{n}{2M_0}-1}.
\end{eqnarray*}
This yields (\ref{exp00}).

It remains to consider the case of $|V_x|>n$. Since $C_x$ is a
$*-$connected component containing $\widetilde{z}$ in $B_{\mathbb{Z}
}'(n(s))$, by a Peierls argument (Lemma 9.3 in \cite{p10}), for all
$k$, the number of $*-$ connected subsets of $B_{\mathbb{Z}
}'(n(s))$ of cardinality $k$ containing $\widetilde{z}$ is at most
$\gamma^k$. Let $c_2\geq e^2 (2M_0)^d \lambda$. If $|C_x|<k$ and
$|V_x|\geq c_2 k+1$, then for at least one of these subsets of
$B_{\mathbb{Z} }'(n(s))$ the union of the associated boxes $B_z$
contains at least $c_2 k$ points of $\mathcal{H}_{\lambda}$.
Therefore, by Lemma 1.2 in \cite{p10}, we have
\begin{eqnarray}\label{exp24}
\begin{aligned}
P[\{|C_x|<k\}\cap\{|V_x|\geq c_2 k+1\}] &< \gamma^k
P \left [Po \left (k(2M_0)^d\lambda \right )\geq c_2 k \right ]\\
&\leq \gamma^k \exp \left \{-\left (\frac{c_2 k}{2} \right ) \log
\left ( \frac{c_2}{(2M_0)^d \lambda } \right ) \right \}.
\end{aligned}
\end{eqnarray}
So if $c_2$ is chosen large enough, this probability decays
exponentially in $k$.

Set $\beta:=(2d)^{-1} (1- (\frac{2}{3})^{1/d})$. By (\ref{exp22_1})
and (\ref{exp23}), we have
\begin{eqnarray*}
P[\{\mbox{diam}(V_x) \leq s/2 \} \cap \{ |C_x| \geq k \} ] \leq
P \left [|D_{ext}C_x| \geq \beta k^{(d-1)/d} \right ]<
(\frac{2}{3})^{\beta k^{(d-1)/d}}.
\end{eqnarray*}
Combined with (\ref{exp24}), this gives (\ref{exp01}).
\end{proof}

\iffalse

\begin{lem}\label{expectation}
Suppose $d\geq 2$ and $\lambda>\lambda_c$. Then
\begin{eqnarray}
0<\sum\limits_{n=1}^{\infty}nP(|V_0|=n)<\infty.
\end{eqnarray}
\end{lem}
\begin{proof}
By Lemma \ref{exponent}, there exist two constants $c>0$ and
$n_0>0$, such that for all $n>n_0$,
\begin{eqnarray*}
n^{-(d-1)/d}\log P( n\leq |V_0|<\infty)<-c.
\end{eqnarray*}
From this it can be deduced that
\begin{eqnarray*}
\sum\limits_{n=n_0+1}^{\infty}P(n\leq |V_0|<\infty)<
\sum\limits_{n=n_0+1}^{\infty}e^{-cn^{(d-1)/d}}<\infty.
\end{eqnarray*}
Therefore, we have
\begin{eqnarray}
\begin{aligned}
\sum\limits_{n=1}^{\infty}nP(|V_0|=n)&=\sum\limits_{n=1}^{\infty}P(n\leq
|V_0|<\infty)\\
&\leq n_0+\sum\limits_{n=n_0+1}^{\infty}P(n\leq
|V_0|<\infty)\\
&< \infty.
\end{aligned}
\end{eqnarray}
In the following we prove that $P(1\leq |V_0|<\infty)>0$.
\end{proof}

Define $\widetilde{\tau}=\widetilde{\tau}
(\lambda):=\sum\limits_{n=1}^{\infty}nP(|V_0|=n)$.

\fi

For $x\in B(s)$ and $0<a\leq 1$, define the box
\begin{eqnarray*}
B_i(x,a):=x\oplus
\big(\underbrace{[0,1]\times\cdots\times[0,1]}\limits_i \times
\underbrace{[0,a]\times\cdots\times[0,a]}\limits_{d-i}\big).
\end{eqnarray*}
Also, for any region $R\subseteq B(s)$, define
\begin{eqnarray*}\label{prop1_1}
D(R)=D(R,s):=\max_{2\leq j \leq M,\chi_{j}(R)=1} \mbox{diam} (C_j).
\end{eqnarray*}

\begin{lem}\label{prop1}
 Suppose $d\geq 2$ and
$\lambda>\lambda_c$. There exist constants $c>0$ and
$n_0>0$, such that if $x\in B(s)$, $a\in(0,1]$ and $n>n_0$ then
\begin{eqnarray}\label{prop1_000}
P[D(B_i(x,a))\geq n]<e^{-c n},
\end{eqnarray}
and
\begin{eqnarray}\label{prop1_00}
P[\xi(B_i(x,a))\geq n]<\exp \left (-c n^{(d-1)/d} \right )+e^{-c
s}.
\end{eqnarray}
\end{lem}

\begin{proof}
Let $W_1$ denote the number of the connected components which
intersect with  $B_i(x,a)$, and have metric diameter not greater
than $s/2$ but not smaller than $n$. By Markov's inequality,
\begin{eqnarray}\label{prop1_02}
P \left [\left \{D(B_i(x,a)) \geq n \right \} \cap \left \{
D(B_i(x,a)) \leq s/2 \right \} \right ]\leq P[W_1>0] \leq
E[W_1].
\end{eqnarray}
By Palm theory for Poisson process and Lemma \ref{exponent2}, if $n>n_0$ then
\begin{eqnarray}\label{prop1_03}
\begin{aligned}
E[W_1]&=\lambda\int_{B_i(x,a)}  P \left [ \left \{
\mbox{diam}(V_x(s)) \geq n \right \} \cap \left \{
\mbox{diam}(V_x(s)) \leq s/2 \right \} \right ]dx\\
&<\lambda a^{d-i} e^{-c n}.
\end{aligned}
\end{eqnarray}
Also, $C_i$ ($2\leq i \leq M$) is not the largest component of
$G(\mathcal{H}_{\lambda,s};1)$, then by Proposition 10.13 in
\cite{p10}, there exist constants $c_1>0$ and $s_1>0$, such that if
$s>s_1$ then
\begin{eqnarray}\label{prop1_04}
&& P \left [ D(B_i(x,a)) > s/2  \right ] <e^{-c_1 s} .
\end{eqnarray}
Together with (\ref{prop1_02}), (\ref{prop1_03}) and
(\ref{prop1_04}), we obtain
\begin{eqnarray*}
P[D(B_i(x,a))\geq n]<e^{-c n}+e^{-c_1 s}.
\end{eqnarray*}
Since $P[D(B_i(x,a))>s]=0$, thus (\ref{prop1_000}) follows.

Note that $B_i(x,a)$ contains at most $2^d$ connected components.
Thus, if $\xi(B_i(x,a)) \geq n$, by the definition of $\xi(\cdot)$,
there exists at least one component intersecting with $B_i(x,a)$
such that it contains no less than $2^{-d}n$ points. Let $W_2$ be
the number of the connected components which intersect with
$B_i(x,a)$, and have more than $2^{-d}n$ elements and not larger
than $s/2$ metric diameter. With the similar argument as
(\ref{prop1_02}) and (\ref{prop1_03}), we get if $n>n_0$ then
\begin{eqnarray*}
&&P \left [\left \{ \xi(B_i(x,a)) \geq n \right \} \cap \left \{
D(B_i(x,a)) \leq s/2 \right \} \right ] \leq E[W_2]\\
&&~~=\lambda\int_{B_i(x,a)}  P \left [ \left \{ |V_x(s)| \geq
2^{-d}n \right \} \cap \left \{ \mbox{diam}(V_x(s)) \leq
s/2 \right \} \right ]dx\\
&&~~<\lambda a^{d-i} \exp\left(-c2^{-d} n\right),
\end{eqnarray*}
together with (\ref{prop1_04}) this gives (\ref{prop1_00}).
\end{proof}

Let real numbers $s_1>2$ and $s_2>2$ be given. Let points
$x=(x_1,x_2,\ldots,x_d)\in [0,s_1/2]^d$ and
$\widetilde{x}=(\widetilde{x}_1,\widetilde{x}_2,\ldots,\widetilde{x}_d)\in
[0,s_2/2]^d$ be given. For all $1\leq j \leq d$, define
\begin{eqnarray*}
N_{x,\widetilde{x}}^j(s_1,s_2):=\left\{
\begin{array}{ll}
\min(s_1,s_2)-x_j-1, & \mbox{if } x_j=\widetilde{x}_j, \\
\min(x_j,\widetilde{x}_j,s_1-x_j-1,s_2-\widetilde{x}_j-1), &
\mbox{otherwise},
\end{array}
\right.
\end{eqnarray*}
and let
\begin{eqnarray}\label{prop1_1}
N_{x,\widetilde{x}}(s_1,s_2):=\min_{1\leq j\leq d}\lfloor
N_{x,\widetilde{x}}^j(s_1,s_2)\rfloor.
\end{eqnarray}

\begin{lem}\label{argument}
Let us assume $d \geq 2$, $\lambda>\lambda_c$, $1\leq i \leq d$ and
$0<a\leq 1$. There exist constants $c>0$ and
$n_0>0$, such that if  $x\in [0,s_1/2]^d$,
$\widetilde{x}\in [0,s_2/2]^d$ and
$N_{x,\widetilde{x}}(s_1,s_2)>n_0$ then
\begin{eqnarray*}
&&\left| E\left[\xi(B_i(x,a),s_1)\right]-
E\left[\xi(B_i(\widetilde{x},a),s_2)\right]\right|<\exp\left(-c
N_{x,\widetilde{x}}(s_1,s_2) \right).
\end{eqnarray*}
\end{lem}
\begin{proof}
Let $B'(s_2):=B(s_2)\oplus\{x-\widetilde{x}\},$ and let
$\widetilde{C}_1,\widetilde{C}_2,\ldots,\widetilde{C}_{\widetilde{M}}$
denote the components of $G(\mathcal{C}_{\infty}\cap B'(s_2);1),$
taking in order of decreasing order. For any region $R\subseteq
B'(s_2)$ and $2\leq i \leq \widetilde{M}$, define
\begin{equation*}
\widetilde{\chi_{i}}(R):=\left\{
\begin{array}{ll}
1, & \mbox{if the out-connect point of $\widetilde{C}_i$ is contained by } R, \\
0, & \mbox{otherwise}.
\end{array}
\right.
\end{equation*}
Let $\widetilde{\xi}(R,s_2):=
\sum_{i=2}^{\widetilde{M}}\widetilde{\chi}_{i}(R)|\widetilde{C}_i|$
and define
\begin{eqnarray*}
\widetilde{D}(R,s_2):=\max_{2\leq j \leq
\widetilde{M},\widetilde{\chi}_{j}(R)=1} \mbox{diam}
(\widetilde{C}_j).
\end{eqnarray*}
According to the ergodicity of Poisson point processes, we can get
\begin{eqnarray}\label{argu_1}
P\left[\widetilde{\xi}\left(B_i(x,a),s_2\right)=k\right]=P\left[\xi\left(B_i(\widetilde{x},a),s_2\right)=k\right],~~~~\forall~k\geq
1.
\end{eqnarray}
\begin{figure}
\centering
\includegraphics[width=3in]{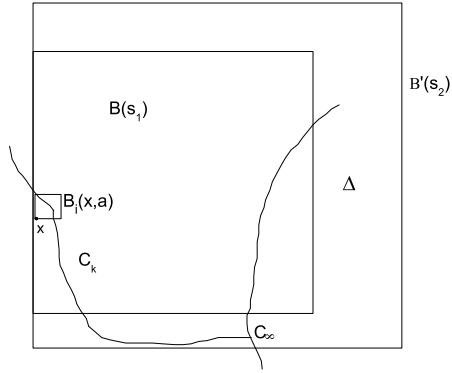}
\caption{If $C_k$ connects with $\mathcal{H}_{\lambda}\cap \Delta$,
 the event of $\xi(B_i(x,a),s_1)\neq
\widetilde{\xi}(B_i(x,a),s_2)$ may happen.}\label{Lemmagraph}
\end{figure}
Let $\Delta:=B(s_1)\cup B'(s_2)-B(s_1)\cap B'(s_2).$ If
$\xi(B_i(x,a),s_1)\neq \widetilde{\xi}(B_i(x,a),s_2)$, then there
exists at least one component among
$C_2,\ldots,C_M,\widetilde{C}_2,\ldots,\widetilde{C}_{\widetilde{M}}$
which connects directly with $\mathcal{H}_{\lambda}\cap \Delta$, see
Figure \ref{Lemmagraph}. For simplicity of exposition, we take
$N=N_{x,\widetilde{x}}(s_1,s_2)$, $\xi_1=\xi(B_i(x,a),s_1)$ and $\xi_2=\widetilde{\xi}(B_i(x,a),s_2)$. Therefore, by (\ref{prop1_000}), if $N>n_0+1$ then
\begin{eqnarray}\label{argu_2}
\begin{aligned}
P\left[\xi_1 \neq \xi_2\right]&\leq P\left[ \left\{D(B_i(x,a),s_1)\geq N-1 \right\} \cup
\left\{\widetilde{D}(B_i(x,a),s_2)\geq N-1 \right\} \right]\\
&< 2e^{-c(N-1)}.
\end{aligned}
\end{eqnarray}
Also,
\begin{eqnarray}\label{argu_2_1}
\begin{aligned}
&P\left[\left\{\xi_1=k\right\} \cap \left\{\xi_2\neq k \right\} \right]+P\left[\left\{\xi_1\neq k\right\} \cap \left\{\xi_2=k \right\} \right]\\
&~~=P\left[\left\{\xi_1=k\right\}\right]+P\left[\left\{\xi_2=k\right\}\right]-2P\left[\left\{\xi_1=k\right\} \cap \left\{\xi_2=k \right\} \right]\\
&~~\geq |P\left[\left\{\xi_1=k\right\}\right]-P\left[\left\{\xi_2=k\right\}\right]|,
\end{aligned}
\end{eqnarray}
so by (\ref{argu_2}) and (\ref{argu_2_1}) we have
\begin{eqnarray}\label{argu_3}
\begin{aligned}
&\sum_{k=1}^{\infty} \left|P\left[\left\{\xi_1=k\right\}\right]-P\left[\left\{\xi_2=k\right\}\right]\right|\\
&~~\leq \sum_{k=1}^{\infty} \left( P\left[\left\{\xi_1=k\right\} \cap \left\{\xi_1\neq \xi_2 \right\} \right]+P\left[\left\{\xi_2=k \right\} \cap \left\{\xi_1 \neq \xi_2 \right\} \right] \right)\\
&~~=P\left[\left\{\xi_1\geq 1 \right\} \cap \left\{\xi_1\neq \xi_2 \right\} \right]+P\left[\left\{\xi_2\geq 1 \right\} \cap \left\{\xi_1 \neq \xi_2 \right\} \right]<4e^{-c(N-1)}.
\end{aligned}
\end{eqnarray}
Thus, by (\ref{argu_1}) and (\ref{argu_3}) we can get
\begin{eqnarray}\label{argu_4}
\begin{aligned}
&\left| E\left[\xi_1\right]-E\left[
\xi(B_i(\widetilde{x},a),s_2)\right]\right|=\left| \sum_{n=1}^{\infty}\sum_{k=n}^{\infty}
\left(P\left[\xi_1=k\right]-P\left[
\xi_2=k \right]\right)  \right|\\
&< 4N^{d/(d-1)}e^{-c(N-1)}+ \sum_{n=N^{d/(d-1)}}^{\infty}\left(
P\left[\xi_1\geq n\right]+P\left[
\xi_2\geq n) \right]\right).
\end{aligned}
\end{eqnarray}

In the following we estimate the upper bound of
$\sum_{n=N^{d/(d-1)}}^{\infty} P[\xi_1 \geq n]$.
 Firstly, by (\ref{prop1_00}), for  $N$ large enough, we can obtain
\begin{eqnarray}\label{argu_5}
\begin{aligned}
\sum\limits_{n=N^{d/(d-1)}}^{e^2 \lambda s_1^d} P[\xi_1
\geq n]<\sum\limits_{n=N^{d/(d-1)}}^{e^2 \lambda s_1^d} \exp
\left (-c n^{(d-1)/d} \right )+ e^2 \lambda s_1^d e^{-c
s_1 }.
\end{aligned}
\end{eqnarray}
Set $\alpha:=\exp (-c N)$, then
\begin{eqnarray}\label{argu_6}
\begin{aligned}
&\sum\limits_{n=N^{d/(d-1)}}^{e^2 \lambda s_1^d} \exp \left (-c
n^{(d-1)/d} \right )=\sum\limits_{n=N^{d/(d-1)}}^{e^2 \lambda
s_1^d} \alpha^{(nN^{-d/(d-1)})^{(d-1)/d}}\\
&~~<N^{d/(d-1)}\sum\limits_{k=1}^{\infty} \alpha ^{k^{(d-1)/d}}=
N^{d/(d-1)} \alpha \sum\limits_{k=1}^{\infty}
\alpha^{k^{(d-1)/d}-1}<M N^{d/(d-1)} \alpha, \end{aligned}
\end{eqnarray}
where $M=\sum_{k=1}^{\infty} \exp(-c (k^{(d-1)/d}-1))<\infty$ is a
constant.

Secondly, by Lemma 1.2 in \cite{p10},
\begin{eqnarray}\label{argu_7}
\begin{aligned}
&\sum\limits_{n=e^2 \lambda s_1^d+1}^{\infty} P[\xi_1
\geq n] < \sum\limits_{n=e^2 \lambda s_1^d+1}^{\infty} P[Po(\lambda
s_1^d)\geq
n]\\
&~~ \leq  \sum\limits_{n=e^2 \lambda s_1^d+1}^{\infty} \exp \left (
-\left (\frac{n}{2} \right )  \log \left ( \frac{n}{\lambda s_1^d}
\right )  \right ) < \frac{e^{-(e^2 \lambda s_1^d +1)}}{1-e^{-1}}.
\end{aligned}
\end{eqnarray}
Thus, by (\ref{argu_5}), (\ref{argu_6}) and (\ref{argu_7}),   there exists a
constant $c_1>0,$ such that for large $N$,
\begin{eqnarray}\label{argu_8}
\sum\limits_{n=N^{d/(d-1)}}^{\infty} P[\xi_1 \geq n]<
e^{-c_1 N}.
\end{eqnarray}
Using the ergodicity of Poisson point processes, similarly, we can
get
\begin{eqnarray}\label{argu_9}
\sum\limits_{n=N^{d/(d-1)}}^{\infty} P[\xi_2
\geq n]< e^{-c_1 N}.
\end{eqnarray}
Combining (\ref{argu_4}), (\ref{argu_8}) and
(\ref{argu_9}) gives us the result.
\end{proof}

\begin{lem}\label{limit}
Suppose $d \geq 2$ and $\lambda>\lambda_c$. Let integer $i\in [1,d]$, and constants $a\in(0,1]$  and $x_j\in [0,\infty)$, $1\leq j \leq i$.
Define the point
$$\widetilde{x}_{s,a}=\widetilde{x}_{s,a}(x_1,\ldots,x_i):=\left(x_1,\ldots,x_i,\frac{s}{2}-a,\ldots,\frac{s}{2}-a\right)\in
\mathbb{R}^d,$$ then the limit of
$E[\xi(B_i(\widetilde{x}_{s,a},a))]$ exists and
\begin{eqnarray}\label{lim_new_00}
\lim_{s\rightarrow\infty}E[\xi(B_i(\widetilde{x}_{s,a},a))]=a^{d-i}\lim_{s\rightarrow\infty}E[\xi(B_i(\widetilde{x}_{s,1},1))].
\end{eqnarray}
Also, if $\min_{1\leq j \leq i}\{x_j\}=0$, then $\lim_{s\rightarrow\infty}E[\xi(B_i(\widetilde{x}_{s,a},a))]>0$.
\end{lem}

\begin{proof}
For  $s_1$ and $s_2$ large enough, suppose $s_2>s_1$. By
(\ref{prop1_1}), it is easy to get
$N_{\widetilde{x}_{s_1,a},\widetilde{x}_{s_2,a}}(s_1,s_2)>s_1/2-2.$
Therefore by Lemma \ref{argument} and Cauchy's criterion for
convergence, the limit of $E[\xi(B_i(\widetilde{x}_{s,a},a))]$
exists as $s\rightarrow\infty$.

For any constant $b\in[0,1]$, let
\begin{eqnarray*}
y_{s,b}=y_{s,b}(x_1,\ldots,x_i):=\left(x_1,\ldots,x_i,\frac{s}{2}-1,\ldots,\frac{s}{2}-1,\frac{s}{2}-b\right)\in\mathbb{R}^d.
\end{eqnarray*}
Similarly, by Lemma \ref{argument} and the Cauchy's criterion we have
the limit of $E[\xi(B_{d-1}(y_{s,b},b))]$
exists. Define
$$f_{x_1,\ldots,x_i}(b):=\lim_{s\rightarrow\infty}E[\xi(B_{d-1}(y_{s,b},b))].$$
Since $Leb(B_{d-1}(y_{s,b},b)\cap B_{d-1}(y_{s,1},1-b))=0$, then by
the definition of $\xi$ we have
\begin{eqnarray}\label{lim_new_01}
E\left[\xi(B_{d-1}(y_{s,1},1))\right]=E[\xi(B_{d-1}(y_{s,1},1-b))]+E[\xi(B_{d-1}(y_{s,b},b))].
\end{eqnarray}
By (\ref{prop1_1}), $N_{y_{s,1},y_{s,1-b}}(s,s)>s/2-2$. Using Lemma
\ref{argument} and Cauchy's criterion we have
\begin{eqnarray*}
\lim_{s\rightarrow\infty}E[\xi(B_{d-1}(y_{s,1},1-b))]=\lim_{s\rightarrow\infty}E[\xi(B_{d-1}(y_{s,1-b},1-b))]=f_{x_1,\ldots,x_i}(1-b).
\end{eqnarray*}
Therefore, taking the limits of the both sides on (\ref{lim_new_01}), we
can get
\begin{eqnarray*}
f_{x_1,\ldots,x_i}(1)=f_{x_1,\ldots,x_i}(1-b)+f_{x_1,\ldots,x_i}(b),
\end{eqnarray*}
which indicates that $f_{x_1,\ldots,x_i}(b)=bf_{x_1,\ldots,x_i}(1)$.
With the similar method, we can get
\begin{eqnarray*}
\lim_{s\rightarrow\infty}E[\xi(B_i(\widetilde{x}_{s,a},a))]=a^{d-i}f_{x_1,\ldots,x_i}(1),
\end{eqnarray*}
which gives (\ref{lim_new_00}).

It remains to prove that $\lim_{s\rightarrow\infty}E[\xi(B_i(\widetilde{x}_{s,a},a))]>0$ if $\min_{1\leq j
\leq i}\{x_j\}=0$. For
simplicity of exposition, we restrict ourselves to the case of
$d=2$, and the proof of this result has no essential difficulty when
$d\geq 3$.

Let $\partial B(s)$ denote the boundary of $B(s)$. If $\min_{1\leq j
\leq i}\{x_j\}=0$, then $\widetilde{x}_{s,a}\in\partial B(s)$.
For $x\in B_i(\widetilde{x}_{s,a},a)$, let $d_x$ to be the Euclid
distance from $x$ to $\partial B(s)$, then $0\leq d_x\leq 1$.  Let
$V_x$ denote the connected component containing $x$ of
$G(\mathcal{H}_{\lambda,s}\cup \{x\};1)$.  Firstly, we will show
that there exists a constant $c>0$, such that
\begin{eqnarray}\label{bound_2}
\begin{aligned}
&P\left[\{|V_x|=1\} \cap \{x\in
\mathcal{C}_{\infty}\}\right]\\
&~~\geq c\left[1-\exp\left(\lambda\left(d_x\sqrt{1-d_x^2}-\arccos
d_x\right)\right)\right]p_{\infty}(\lambda).
\end{aligned}
\end{eqnarray} Define
\begin{figure}
\centering
\includegraphics[width=2in]{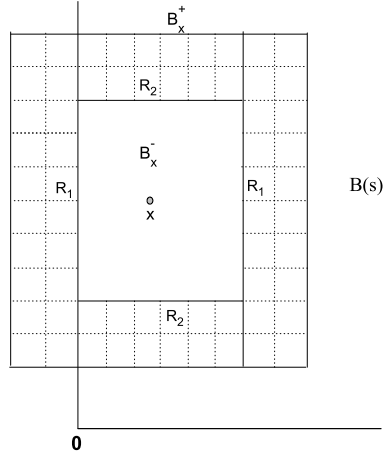}
\caption{The placements of $B_x^-,B_x^+,R_1$ and $R_2$ are
shown.}\label{box2}
\end{figure}
$B_x^-$ to be the rectangle of $(1+d_x)\times 2$ centred at $x$ and
$B_x^+$ to be the rectangle of $(\frac{7}{3}+d_x)\times
\frac{10}{3}$ centred at $x$. Divide the region of $B_x^+\backslash
B_x^-$ into 64 small rectangles with two diffrent sizes: one size
recorded $R_1$ is $\frac{1}{3}\times \frac{1}{3}$, and the other
size recorded $R_2$ is $\frac{1+d_x}{6}\times \frac{1}{3}$, see
Figure \ref{box2}. The number of small rectangles with size $R_1$ is
$40$, and the number of small rectangles with size $R_2$ is $24$.
Define $A_1$ to be the event that each of these 64 small rectangles
includes at least one point of $\mathcal{H}_{\lambda}$. By the
properties of Poisson point processes, we have
\begin{eqnarray}\label{bound_3}
\begin{aligned}
P(A_1)&=\left(1-e^{-\lambda/9}\right
)^{40}\cdot\left(1-e^{-\lambda(1+d_x)/18}\right )^{24}\\
&\geq \left(1-e^{-\lambda/9}\right
)^{40}\cdot\left(1-e^{-\lambda/18}\right )^{24}.
\end{aligned}
\end{eqnarray}
If $A_1$ happens, there exists a connected component in
$B_x^+\backslash B_x^-$ which contains all the points in these small
rectangles. Also, for any point in $\mathbb{R}^d \backslash B_x^-$
which can connect directly with a point in $B_x^-$, it must connect
directly with this connected component. Let $A_2$ denote the event
that there exists at least one point in $B_x^+\backslash B_x^-$
contained by $\mathcal{C}_{\infty}$. So according to above
discussion, the event $A_1\cap A_2$ is independent with the
distribution of the points of $\mathcal{H}_{\lambda}$ in $B_x^-$.
Therefore,
\begin{eqnarray}\label{bound_4}
P(A_1\cap A_2)=P(A_1)P(A_2|A_1) \geq P(A_1) p_{\infty}(\lambda).
\end{eqnarray}
Denote $A_3$ to be the event that there exists at least one point of
$\mathcal{H}_{\lambda}$ in $B(x;1) \cap B(s)^c$, where $B(x;1)$
denotes the $d-dimensional$ unit ball centred at point $x$. By the
properties of Poisson point processes it can be computed that
\begin{eqnarray}\label{bound_5}
P(A_3)=1-\exp\left(\lambda\left(d_x\sqrt{1-d_x^2}-\arccos
d_x\right)\right).
\end{eqnarray}
Because $A_3$ and $A_1\cap A_2$ are both increasing events in
$G(\mathcal{H}_{\lambda};1)$, by FKG inequality (Theorem 2.2 in
\cite{p12}) we have
\begin{eqnarray}\label{bound_6}
P(A_3\cap A_1 \cap A_2) \geq P(A_3)P(A_1 \cap A_2).
\end{eqnarray}
If the event $A_3\cap A_1 \cap A_2$ happens, it must be true that
$x\in \mathcal{C}_{\infty}$. Also, the event $A_3$ is independent
with the distribution of the points of $\mathcal{H}_{\lambda}$ in
$B_x^-$, so we have
\begin{eqnarray}\label{bound_6_temp}
\begin{aligned}
P\left[\{|V_x|=1\} \cap \{x\in \mathcal{C}_{\infty}\}\right] &\geq
P[A_3\cap A_1 \cap A_2 \cap \{\mathcal{H}_{\lambda} \cap B_x^-
= \emptyset\}]\\
&=e^{-2(1+x)\lambda}P(A_3\cap A_1 \cap A_2)\\
&\geq e^{-4\lambda}P(A_3\cap A_1 \cap A_2).
\end{aligned}
\end{eqnarray}
Set $c:=e^{-4\lambda}\cdot \left(1-e^{-\lambda/9}\right
)^{40}\cdot\left(1-e^{-\lambda/18}\right )^{24}$, together with
(\ref{bound_3}), (\ref{bound_4}), (\ref{bound_5}), (\ref{bound_6})
and (\ref{bound_6_temp}) we can get (\ref{bound_2}).

Let $W$ denote the number of the points of $\mathcal{H}_{\lambda}
\cap B_i(\widetilde{x}_{s,a},a)$ which belong to
$\mathcal{C}_{\infty}$ but are isolated in $B(s)$. By the definition
of $\xi(B_i(\widetilde{x}_{s,a},a))$ and Palm theory for Poisson
processes, we have
\begin{eqnarray*}
E[\xi(B_i(\widetilde{x}_{s,a},a))]\geq E[W]=\lambda
\int_{B_i(\widetilde{x}_{s,a},a))} P\left[\{|V_x|=1\} \cap \{x\in
\mathcal{C}_{\infty}\}\right] dx.
\end{eqnarray*}
Combining this with (\ref{bound_2}), we can get
$E[\xi(B_i(\widetilde{x}_{s,a},a))]>\frac{1}{2}c\left(1-e^{(1-\pi)\lambda/4}\right)\lambda
p_{\infty}(\lambda).$ Our result follows.

\end{proof}

\begin{proof}[Proof of Theorem \ref{t1}]
For simplicity of exposition, we shall prove (\ref{order_t1_00}) only in the case of $d=3$, and this proof
has no essential difficulty in the case of $d=2$ or $d\geq 4$.

Let $\eta_{ij}(s):=E\left[\xi\left([0,1]\times [i,i+1]\times [j,j+1],s \right)\right]$ and take $n=\lfloor \frac{s}{2} \rfloor$. By symmetry we have $\eta_{ij}(s)=\eta_{ji}(s)$, and therefore
\begin{eqnarray}\label{t1_1}
\begin{aligned}
E\left[\xi\left([0,1]\times[0,n]^2\right)\right]&=\sum_{i=0}^{n-1}\sum_{j=0}^{n-1}\eta_{ij}(s)\\
&=\eta_{00}(s)+\sum_{k=1}^{n-1}\left(2\sum_{i=0}^{k-1}\eta_{ik}(s)+\eta_{kk}(s) \right).
\end{aligned}
\end{eqnarray}
Set $$a_1(s):=\eta_{00}(s)+\sum_{k=1}^{n-1}\left(2\sum_{i=0}^{k-1}\left(\eta_{ik}(s)-\eta_{i,n-1}(s)\right)+\eta_{kk}(s)-\eta_{k,n-1}(s) \right),$$
then for large $s$ and $s_2$ satisfying $s_2>s$, by Lemma \ref{argument} we have
\begin{eqnarray}\label{t1_2}
\begin{aligned}
&|a_1(s)-a_1(s_2)|<2n^2 e^{-cs/2}\\
&~~~~~~~~+\sum_{k=n}^{n_2-1}\left(2\sum_{i=0}^{k-1}|\eta_{ik}(s_2)-\eta_{i,n_2-1}(s_2)|+|\eta_{kk}(s_2)-\eta_{k,n_2-1}(s_2)| \right)\\
&<2n^2 e^{-cs/2}+ \sum_{k=n}^{n_2-1}\left(2\sum_{i=0}^{k-1}e^{-ck}+e^{-ck}\right)=o\left(e^{-cs/3}\right),
\end{aligned}
\end{eqnarray}
where $n_2=\lfloor\frac{s_2}{2} \rfloor$ and $c$ is the same constant appearing in Lemma \ref{argument}. Then by Cauchy's criterion
the limit of $a_1(s)$ exists.

Define the point $y_i=(0,i,n)\in \mathbb{R}^3$. For any $i\in [0,n-1]$ and large $s$, using Lemmas \ref{argument} and \ref{limit} we can get
\begin{eqnarray}\label{t1_3}
\begin{aligned}
&\big|E\left[\xi\left(B_{2}\left(y_{i},{\textstyle  \frac{s}{2}}-n\right)\right)\right]-({\textstyle  \frac{s}{2}}-n)\eta_{i,n-1}(s)\big|\\
&\leq \big|E\left[\xi\left(B_{2}\left(y_{i},{\textstyle  \frac{s}{2}}-n\right)\right)\right]-({\textstyle  \frac{s}{2}}-n)E\left[\xi\left([0,1]\times[i,i+1]\times[{\textstyle  \frac{s}{2}}-1,{\textstyle  \frac{s}{2}}]\right)\right]\big|\\
&~~~~~~~~+({\textstyle  \frac{s}{2}}-n)\big|E\left[\xi\left([0,1]\times[i,i+1]\times[{\textstyle  \frac{s}{2}}-1,{\textstyle  \frac{s}{2}}]\right)\right]-\eta_{i,n-1}(s)\big|\\
&=o\left(e^{-cs/3}\right).
\end{aligned}
\end{eqnarray}
Similarly, we can get
\begin{eqnarray}\label{t1_4}
E\left[\xi\left([0,1]\times\left[n,\frac{s}{2}\right]^2\right)\right]=\left(\frac{s}{2}-n\right)^2\eta_{n-1,n-1}(s)+o\left(e^{-cs/3}\right).
\end{eqnarray}
We recall that
$R_0=[0,1]\times [0,s/2]^2$, then together with (\ref{t1_1}), (\ref{t1_2}), (\ref{t1_3}) and (\ref{t1_4}),
\begin{eqnarray}\label{t1_5}
\begin{aligned}
E\left[\xi\left(R_0\right)\right]&=E\left[\xi\left([0,1]\times[0,n]^2\right)\right]+2\sum_{i=0}^{n-1}E\left[\xi\left(B_{2}\left(y_{i},\frac{s}{2}-n\right)\right)\right]\\
&~~~~~~~~~~~~+E\left[\xi\left([0,1]\times\left[n,\frac{s}{2}\right]^2\right)\right]\\
&=\sum_{k=1}^{n-1}\left(2\sum_{i=0}^{k-1}\eta_{i,n-1}(s)+\eta_{k,n-1}(s) \right)+\left(s-2n\right)\sum_{i=0}^{n-1}\eta_{i,n-1}(s)\\
&~~~~~~~~~~~~+\left(\frac{s}{2}-n\right)^2\eta_{n-1,n-1}(s)+a_1+o\left(e^{-cs/3}\right),
\end{aligned}
\end{eqnarray}
where $a_1:=\lim_{s\rightarrow\infty}a_1(s)$. Let $b_i(s):=\eta_{i,n-1}(s)-\eta_{n-1,n-1}(s)$, then by (\ref{t1_5}) we have
\begin{eqnarray}\label{t1_6}
\begin{aligned}
E\left[\xi\left(R_0\right)\right]
&=\left(\frac{s^2}{4}-1\right)\eta_{n-1,n-1}(s)+\sum_{k=1}^{n-1}\left(2\sum_{i=0}^{k-1}b_i(s)+b_k(s) \right)\\
&~~~~~~~~~~~~+\left(s-2n\right)\sum_{i=0}^{n-1}b_i(s)+a_1+o\left(e^{-cs/3}\right)\\
&=\left(\frac{s^2}{4}-1\right)\eta_{n-1,n-1}(s)+s\sum_{i=0}^{n-2}b_i(s)-2b_0(s)-\sum_{i=1}^{n-2}(2i+1)b_i(s)\\
&~~~~~~~~~~~~+a_1+o\left(e^{-cs/3}\right).
\end{aligned}
\end{eqnarray}
Set
\begin{eqnarray*}
a_2(s):=\sum_{i=0}^{n-2}b_i(s)~~~\mbox{and}~~~a_3(s):=2b_0(s)+\sum_{i=1}^{n-2}(2i+1)b_i(s).
\end{eqnarray*}
With the similar argument as (\ref{t1_3}), we can get that the exist constants $a_2$ and $a_3$ such that
\begin{eqnarray*}
|a_2(s)-a_2|<3ne^{-cs/2}~~~\mbox{and}~~~|a_3(s)-a_3|=o\left(e^{-cs/3}\right).
\end{eqnarray*}
Also, by Lemmas \ref{argument} and the Cauchy's criterion, there exists a constant $a_0>0$ such that
$$|\eta_{n-1,n-1}(s)-a_0|<e^{-c(n-1)}.$$ Taking $a_0,a_2$ and $a_3$ into (\ref{t1_6}) we have
\begin{eqnarray*}
E\left[\xi\left(R_0\right)\right]=\left(\frac{s^2}{4}-1\right)a_0+s a_2-a_3+a_1+o\left(e^{-cs/3}\right).
\end{eqnarray*}
with the similar argument as above, there exist constants $a_4, a_5, a_6$ and $a_7$, such that
\begin{eqnarray*}
E\left[\xi\left(R_1\right)\right]=\frac{s^2}{4}a_0+s a_4+a_5+o\left(e^{-cs/3}\right),
\end{eqnarray*}
and
\begin{eqnarray*}
E\left[\xi\left(R_2\right)\right]=\frac{s^2}{4}a_0+s a_6+a_7+o\left(e^{-cs/3}\right).
\end{eqnarray*}
Combined these with (\ref{comm_6}), (\ref{total_1}) and Lemma \ref{temp1}, (\ref{order_t1_00}) has been deduced,
where $\tau_1=6a_0>0$.

With the results of Theorem 10.22 and Theorem 11.16 (which shows that
$\delta>0$) in \cite{p10}, (\ref{order_t1_00}) is followed by
(\ref{order_t1_01}).
\end{proof}

\begin{proof}[Proof of Theorem \ref{t2}]
Given the discussion in the proof of Theorem 11.16 in \cite{p10},
(2.45) in \cite{p10} is followed by
\begin{eqnarray*}
\left(n/\lambda\right)^{-1/2}
\left(L_1\left(G\left(\mathcal{X}_n;(n/\lambda)^{-1/d}\right)\right)-E[L_1(G(\mathcal{H}_{\lambda,s};1))]\right)
 \xrightarrow{D} \mathcal {N}(0,\delta^2),
\end{eqnarray*}
where $s=(n/\lambda)^{1/d}$. Combining this and (\ref{order_t1_00}) our
result follows.
\end{proof}

\section{Some Applications}
Our method used in the proof of Theorem \ref{t1} can be applied to
estimate the expectation of many other random variables restricted
to a box $B$ as $B$ becomes large, for example, the size of the
biggest \emph{open cluster} for percolation, the coverage area of
the largest component for Poisson Boolean model, the number of open
clusters or connected components for percolation and Poisson Boolean
model, the number of open clusters or connected components with
order $k$ for percolation and Poisson Boolean model, the final size
of a spatial epidemic mentioned in \cite{p10} and so on. We will
give the similar results as  Theorem \ref{t1} for the size of the
biggest open cluster and the number of open clusters for site
percolation but the method can be adapted to bond percolation.

Following Chapter 1 of \cite{p13}, let
$\mathbb{L}^d=(\mathbb{Z}^d,\mathbb{E}^d)$ denote the integer
lattice with vertex set $\mathbb{Z}^d$ and edges $\mathbb{E}^d$
between all vertex pairs at an $l_1$-distance of 1. For $d\geq 2$ we
take $X =(X_x,x\in \mathbb{Z}^d)$ to be a family of i.i.d. Bernoulli
random variables with parameter $p\in (0,1)$. Sites
$x\in\mathbb{Z}^d$ with $X_x = 1(0)$ are denoted \emph{open}
(\emph{closed}). The corresponding probability measure of on
$\{0,1\}^{\mathbb{Z}^d}$ is denoted by $P_p$. The open clusters are
denoted by the connected components of the subgraph of
$\mathbb{L}^d$ induced by the set of open vertices. Let
$C_{\textbf{0}}$ denote the open cluster containing the origin. The
percolation probability is $\theta(p)=P_p(|C_{\textbf{0}}|=\infty)$
 and the critical probability is
 $p_c=p_c(d):=\sup\{p:\theta(p)=0\}.$
 It is well known \cite{p13} that $p_c\in(0,1)$. If $p>p_c$, by Theorem 8.1 in \cite{p13}, with probability $1$
there exists exactly one infinite open cluster
$\mathcal{C}_{\infty}$.

Given integer $n>0$, we denote by open clusters in $B(n)$ the
connected components of the subgraph of the integer lattice
$\mathbb{L}^d$ induced by the set of open vertices lying in $B(n)$.
%By the size of an open cluster or an open cluster in $B(n)$ we mean
%the number of vertices it contains. The biggest open cluster in
%$B(n)$ is the one with the greatest size (not necessarily unique). A
Similar results as Theorem \ref{t1} concerned with the order of the
biggest open cluster in $B(n)$ can be given as follows.

\begin{thm}\label{t3}
Suppose $d\geq 2$ and $p\in (p_c,1)$. Let $H(X;B(n))$ be the order
of the biggest open cluster in $B(n)$. Then there exist constants
$c=c(d,p)>0$ and $\tau_i=\tau_i(d,p)$, $1\leq i \leq d$, with $\tau_1>0$,
such that for all large enough $n$,
\begin{eqnarray}\label{t3_00}
E_p[H(X;B(n-1))]=\theta(p)n^d-\sum_{i=1}^d\tau_i n^{d-i}+o\left(e^{-c n}\right).
\end{eqnarray}
Also,  there exists a constant $\sigma=\sigma(d,p) > 0$, such that
\begin{eqnarray}\label{t3_01}
H(X;B(n-1))n^{-d/2}-  \theta(p) n^{d/2}+\sum_{i=1}^{\lfloor\frac{d}{2}\rfloor}\tau_i n^{d/2-i}
 \xrightarrow{D} \mathcal {N}(0,\sigma^2)
\end{eqnarray}
as $n\rightarrow\infty$.
\end{thm}
\begin{proof}
Similar to the above, $E_p[|\mathcal {C}_{\infty}\cap
B(n-1)|]=\theta(p)n^d$. Let $C_1,C_2,...,C_M$ denote the components
of $\mathcal {C}_{\infty}\cap B(n-1)$, taken in a decreasing order.
Let $L(n-1)=B(n-1)\backslash[1,n-2]^d$. For any $2\leq i \leq M$,
since $C_i\subset \mathcal {C}_{\infty}$, therefore there exists at
least one point in $L(n-1)\cap C_i$ which connects to $\mathcal
{C}_{\infty}$ directly; we choose the smallest one according to the
lexicographic ordering on $\mathbb{Z}^d$ as the $out-connect$
$point$. For any $x\in \mathbb{Z}^d \cap L(n-1)$, define
\begin{equation*}
\xi(x):=\left\{
\begin{array}{ll}
|C_i|, &\mbox{if there exists } i\in[2,M] \mbox{ such that $x$ is the out-connect point of $C_i$}, \\
0, & \mbox{otherwise},
\end{array}
\right.
\end{equation*}
Also, for integer $j\in[0,d-1]$, let
\begin{eqnarray*}
R_j:=\left([0,1]\times [0,n-1]^{d-1-j} \times [1,n-2]^j \right) \cap \mathbb{Z}^d,
\end{eqnarray*}
then $$E\left[\sum_{i=2}^M |C_i|\right]=\sum_{x\in \mathbb{Z}^d \cap L(n-1)
}E[\xi(x)]=2\sum_{j=0}^{d-1}\sum_{x\in R_j
}E[\xi(x)].$$ With the similar process as the proof of Theorem
\ref{t1},  (\ref{t3_00}) can be deduced, where $$\tau_1=2d\lim_{n\rightarrow\infty}E\left[\xi\left(\left(0,\lfloor\frac{n}{2}\rfloor,\ldots,\lfloor\frac{n}{2}\rfloor\right)\right)\right]>0.$$ Using Theorem 3.2
in \cite{penrose2001}, (\ref{t3_01}) follows.
\end{proof}

Following Chapter 1.5 of \cite{p13}, we define the \emph{number of
open clusters per vertex} by
$$\kappa(p)=E_p(|C_{\textbf{0}}|^{-1})=\sum_{n=1}^{\infty}\frac{1}{n}P_p(|C_{\textbf{0}}|=n),$$
with the convention that $1/\infty=0$. Similar results as Theorem
\ref{t1} concerning with the number of the open clusters in $B(n)$
can also be given as follows.

\begin{thm}\label{t4}
Suppose $d\geq 2$ and $p\in (0,p_c)\cup (p_c,1)$. Let $H(X;B(n))$ be
the number of the open clusters in $B(n)$. Then there exist
constants $c=c(d,p)>0$ and $\tau_i=\tau_i(d,p)>0$, $1\leq i \leq d$, with $\tau_1>0$,
 such that for all large enough $n$,
\begin{eqnarray}\label{t4_00}
E_p[H(X;B(n-1))]=\kappa(p)n^d+\sum_{i=1}^d \tau_i n^{d-i}+o\left(e^{-cn}\right).
\end{eqnarray}
Also,  there exists a constant $\sigma=\sigma(d,p) > 0$, such that
\begin{eqnarray}\label{t4_01}
H(X;B(n-1))n^{-d/2}-  \kappa(p) n^{d/2}- \sum_{i=1}^{\lfloor \frac{d}{2} \rfloor}\tau_i n^{d/2-i}
 \xrightarrow{D} \mathcal {N}(0,\sigma^2)
\end{eqnarray}
as $n\rightarrow\infty$. Moreover, for any constant $\varepsilon\in(0,d/2)$,
\begin{eqnarray}\label{t4_03}
\begin{aligned}
&P_p\left(\frac{H(X;B(n-1))-\kappa(p)n^d-\sum_{i=1}^d \tau_i n^{d-i}}{\mbox{Var}(H(X;B(n-1)))}\leq x \right)\\
&=\int_{-\infty}^x \frac{1}{\sqrt{2\pi}} e^{-y^2/2}dy+o\left(n^{-\frac{d}{2}+\varepsilon}  \right),
\end{aligned}
\end{eqnarray}
where Var$(\cdot)$ denotes the variance.
\end{thm}
\begin{proof}
Let $L(n-1)=B(n-1)\backslash[1,n-2]^d$. For any $x\in B(n-1)\cap
\mathbb{Z}^d$, let $C_x$ denote the open cluster including $x$, and
let $C_x(B(n-1))$ denote the open cluster including $x$ in $B(n-1)$.
Then $C_x(B(n-1))\subseteq C_x$. For all open clusters $C$ in
$B(n-1)$, if $C\cap L(n-1)\neq \emptyset,$ according to the
lexicographic ordering on $\mathbb{Z}^d$ we choose the smallest
element of $C\cap L(n-1)$ as the \emph{indicated vertex} of $C$.
 For any $x\in \mathbb{Z}^d \cap L(n-1)$,  define
\begin{equation*}
\xi(x,B(n-1)):=\left\{
\begin{array}{ll}
1-\frac{|C_x(B(n-1))|}{|C_x|}, &\mbox{if $x$ is the idicated vertex of $C_x(B(n-1))$}, \\
0, & \mbox{otherwise}.
\end{array}
\right.
\end{equation*}
Noted that for any $y\in  \mathbb{Z}^d \cap B(n-1)$, $$\sum_{x\in
C_y(B(n-1))} \left(|C_y(B(n-1))|^{-1}-|C_y|^{-1}\right)=1-
\frac{|C_y(B(n-1))|}{|C_y|},$$ then by (4.7) in \cite{p13}, we have
\begin{eqnarray}\label{t4_1}
\begin{aligned}
H(X;B(n-1))&=\sum_{x\in \mathbb{Z}^d \cap B(n-1)}
|C_x(B(n-1))|^{-1}\\
&= \sum_{x\in \mathbb{Z}^d \cap B(n-1)} |C_x|^{-1}+ \sum_{x\in
\mathbb{Z}^d \cap L(n-1)} \xi(x,B(n-1)).
\end{aligned}
\end{eqnarray}
Therefore, take the expectation for the both sides of (\ref{t4_1}),
we can get $$E_p[H(X;B(n-1))]=\kappa(p)n^d + \sum_{x\in \mathbb{Z}^d
\cap L(n-1)} E_p[\xi(x,B(n-1))].$$

Suppose $1\leq i \leq d$ and  $x_j\in [0,K/2-1]\cap \mathbb{Z}$ for
$1\leq j \leq i$. For large integers $n_1,n_2$, let
$x=(x_1,\ldots,x_i,\lfloor \frac{n_1}{2} \rfloor, \ldots,\lfloor
\frac{n_1}{2} \rfloor) \in \mathbb{Z}^d$ and
$\widetilde{x}=(x_1,\ldots,x_i,\lfloor \frac{n_2}{2} \rfloor,
\ldots,\lfloor \frac{n_2}{2} \rfloor ) \in \mathbb{Z}^d$. Set
$\widetilde{B}(n_2):=B(n_2)\oplus\{x-\widetilde{x}\}.$ Since $\xi$
is stationary under translations of the lattice $\mathbb{L}^d$, then
$\xi(\widetilde{x},B(n_2))$ and $\xi(x,\widetilde{B}(n_2))$ have the
same distribution function. However, let
$n_0=\min\{\lfloor\frac{n_1}{2}\rfloor,\lfloor\frac{n_2}{2}\rfloor\}$,
by the definition of $\xi$ we have
\begin{eqnarray*}
&&P_p\left[\xi(x,B(n_1))\neq \xi(x,\widetilde{B}(n_2)) \right]=P_p\left[\xi(x,B(n_1))\neq \xi(x,\widetilde{B}(n_2)),C_x\neq C_\infty \right]\nonumber\\
&&~~\leq P_p\left[\mbox{diam}(C_x)\geq n_0,C_x\neq C_\infty\right]<
e^{-cn_0},
\end{eqnarray*}
where the last inequality follows from Theorem 6.1 of \cite{p13} for
$p<p_c$ and Theorem 8.18 of \cite{p13} for $p>p_c$ respectively.
Thus,
\begin{eqnarray*}
&&\left|E_p\left[\xi(x,B(n_1))\right]-E_p\left[\xi(\widetilde{x},B(n_2))\right]\right|\\
&&~~\leq \sum_{t} t \left|
P_p\left[\xi(x,B(n_1))=t\right]-P_p\left[\xi(x,\widetilde{B}(n_2))=t\right]
\right|\\
&&~~\leq \sum_{t} \left( P_p\left[\xi(x,B(n_1))=t,\xi(x,B(n_1))\neq
\xi(x,\widetilde{B}(n_2))
\right] \right.\\
&&~~~~~~~~~~\left.+P_p\left[\xi(x,\widetilde{B}(n_2))=t,\xi(x,B(n_1))\neq
\xi(x,\widetilde{B}(n_2))\right]
\right)\\
&&~~=2P_p\left[\xi(x,B(n_1))\neq \xi(x,\widetilde{B}(n_2))
\right]<2e^{-c n_0}.
\end{eqnarray*}
Therefore, $\lim_{n\rightarrow\infty}E_p[\xi(x,B(n)]$ exists. In
fact, a similar result as Theorem \ref{argument} can be deduced. Let
\begin{eqnarray*}
\widetilde{\tau}_i(K)={{d}\choose{i}}\sum_{x_j\in [0,K-1]\cup
[n-K,n-1],1\leq j\leq i }
\lim_{n\rightarrow\infty}E_p\left[\xi\left(\left(x_1,\ldots,x_i,\left\lfloor
\frac{n}{2} \right\rfloor,\ldots,\left\lfloor \frac{n}{2}
\right\rfloor\right)\right)\right],
\end{eqnarray*}
and let $\tau_i(K)=\sum_{j=1}^i \widetilde{\tau}_j(K)
{{d-j}\choose{i-j}} (-2K)^{i-j}.$ In a similar way, (\ref{t4_00})
is obtained.

Combining (\ref{t4_00}) with Theorem 3.1 in
\cite{penrose2001}, (\ref{t4_01}) follows immediately.

By Theorem 2.1 in \cite{JJP2010}, Theorem 3.1 in
\cite{penrose2001} and  (\ref{t4_00}), (\ref{t4_03}) can be deduced.
\end{proof}

It is worth noting that our results do have significance for some
practical applications. In fact, the initial motivation of this
paper is to provide theoretical foundation and guidance for the
design of \emph{wireless multihop networks}. The wireless multihop
networks, e.g., vehicular ad hoc networks, mobile ad hoc networks,
and wireless sensor networks, typically consists of a group of
decentralized and self-organized nodes that communicate with each
other in a peer-to-peer manner over wireless channels, and are
increasingly being used in military and civilian applications
\cite{p16}. The large scale wireless multihop networks are usually
formulated by the random geometric graphs, and the size of the
largest component is a fundamental variable for a network, which
plays a key role for the topology control in wireless multihop
networks.  However, this variable can not be described very
precisely by both former theoretic results and even computer
simulations as the scale of the network grows to very large.
Theorem \ref{t1} and  Theorem \ref{t2} provides a precise estimation for this
variable respectively. Using simulations the approximative values of the parameters $p_{\infty}(\lambda)$, $\tau_i$, $\sigma$ and $\delta$ can be
obtained, and thus the expression of the asymptotic size of the largest component can be well established,
which has guiding significance to the topology control in wireless multihop
networks.

% Write the text of your paper using normal LaTeX commands.
% For instance, you can use the `\cite' command~\cite{ref1}.
% When giving citations a numbering system is preferred~\cite{ref2},
% but an author--date system is also acceptable~\cite{ref3}.

%If your paper includes appendices, then precede the first of them by the command
%\appendix
%and then carry on using the \section and \subsection commands, as above.

%\section{The first appendix}

%If you can include EPS (encapsulated postscript) figures in your
%paper, then please use the \begin{figure} . . . \end{figure}
%commands and usual.  If EPS files are unavailable, then please use
%the following command to indicate the approximate position of the
%figures.

%\Fig{Caption.}

\acks This research was Supported by the National Natural Science
Foundation of China under Grants No. 61203141 and 71271204, and the Innovation Program
of the Chinese Academy of Sciences under Grant No. kjcx-yw-s7.

% Reference list
%
% References should be in the following form (or the BibTeX file
% apt.bst should be used):
%
% For a journal:
% Surname, Initial (year). Title of paper. {\em Journal title}
% {\bf Vol,} page--range.
%
% For a book:
% Surname, Initial (year). {\em Book title}. Publisher, Address.
%
% Note the following example of a reference list.

\end{document}